\documentclass[12pt]{amsart}
\usepackage{amsmath,amsthm,latexsym,amscd,amsbsy,amssymb,url}
\usepackage[all]{xypic}
 \relax

\chardef\bslash=`\\ 

\makeatletter
\def\verbatim{\interlinepenalty\@M \@verbatim
  \leftskip\@totalleftmargin\advance\leftskip2pc
  \frenchspacing\@vobeyspaces \@xverbatim}
\makeatother
\newtheorem{thm}{Theorem}[section]

\newtheorem{lem}[thm]{Lemma}

\newtheorem{que}[thm]{Question}

\begin{document}


\centerline{\it In memory of my dear friend Alex Chigogidze}

\par\bigskip
\title
{Ordering A Square}
\author{Raushan ~Z.~Buzyakova}
\email{Raushan\_Buzyakova@yahoo.com}


\begin{abstract}{
We identify a condition on $X$ that guarantees that any finite power of $X$ is homeomorphic to a subspace of a linearly ordered space.} 
\end{abstract}

\maketitle
\markboth{Raushan Z. Buzyakova}{Ordering A Square}
{ }

\section{Introduction}\label{S:intro}
\par\bigskip
To start our discussion, let us agree on terminology. A linearly ordered space, or a LOTS, is one whose topology is generated by open intervals and open rays with respect to some linear order on $X$. A space homeomorphic to a subspace of a LOTS is called a generalized ordered space, or a GO space. It is known (see, for example, \cite{BL}), a Hausdorff space $L$ is a GO space, if  its topology can be generated by a collection of convex sets (not necessarily open) with respect to some linear  order on $L$.  Throughout the paper, we will also refer to LOTS and GO spaces as orderable and suborderable, respectively.

Having a topology compatible with an order is a delicate property that can be destroyed by many standard operations. Under favorable conditions, however,  order-topology ties demonstrate a remarkable resistance to the product operation as demonstrated by zero-dimensional separable metric spaces. A step into a non-metrizable world quickly reveals that zero-dimensionality has to be coupled with very strong properties to achieve a desired resistance of orderability to the product operation. For example, $\omega_1$ is a zero-dimensional LOTS with "cannot be better"  local properties, while $\omega_1\times\omega_1$ has a rather rigid non-orderable structure. One can see that  $\omega_1^2$ is  not sub-orderable by observing that it is not hereditarily normal. Another natural example is Sorgenfrey Line, which is known to be a GO-space. The square of the line is not a GO-space. One explanation is non-normality. Observe, however, that every countable power of the Sorgenfrey Line admits a continuous injection into the irrationals. 

The goal of this paper is to identify a property that may serve as a characterization of spaces with (sub) orderable finite powers. In Theorem \ref{thm:mainpower}, we present a sufficient conditions that may turn out to be a necessary one.  We then observe that the condition in Theorem \ref{thm:mainpower} is a criterion  in the scope of well ordered subspaces (Theorem \ref{thm:ordinal}). We do not know if Theorem \ref{thm:mainpower} can be reversed without narrowing its scope.

In notation and terminology, we will follow \cite{Eng}. Even though we study ordered spaces, we will not need a notation for an open interval. Therefore, we reserve $(a,b)$ to denote the ordered pair. All spaces are assumed Tychonoff. We say that a subset $A\subset X$ separates distinct $x,y\in X$ if $A$ meets $\{x,y\}$ by exactly one element. A family $\mathcal A$ of subsets of $X$ separates distinct $x,y\in X$ if at least one member of the family separates $x$ from $y$.

\section{Results}\label{S:results}
\par\bigskip
To formulate our main result we would like to introduce a property that can be extracted from many arguments leading to orderability of certain spaces.
\par\bigskip\noindent
{\bf Definition of the $P$-number of $X$.}  {\it  The $P$-number of a space $X$  is $|X|$  if $X$ is discrete. Otherwise, the $P$-number of $X$ is the largest cardinal number $\tau$ such that the intersection of any fewer than $\tau$ open sets of $X$ is open.}
\par\bigskip\noindent
Note that the $P$-number of $X$ is well-defined. Indeed, if $X$ is not discrete, then there exists an infinite set $A$ and 
$b\in \bar A\setminus A$. Then the family  $\{X\setminus \{a\}:a\in A\}$  consists of open sets and has non-open intersection. Thus, the $P$-number of $X$ is 
$$
\min\{|\mathcal O|: \mathcal O\ consists\ of \ open \ sets\ and\ has \ nonopen\ intersection\}
$$
Since the minimum is computed over a non-empty set, it exists. In particular, the $P$-number of any non-discrete metric space 
is $\omega$. The $P$-number of $\{\omega_1\}\cup \{\alpha\in\omega_1: \alpha\ is \ isolated\}$  is $\omega_1$.
\par\bigskip
Note that if a space has an $\omega_1$-long convergent sequence and an $\omega$-long convergent sequence, then the square of the space is not suborderable because the product of these two sequences is not hereditarily normal. This and similar structures are eliminated if a space has a $\tau$-discrete basis, where $\tau$ is the $P$-number of the space. We start by identifying a necessary condition for suborderability. We then show that the found property is finitely productive.
\par\bigskip\noindent
\begin{thm}\label{thm:sufficient}
Let $X$ have a $\tau$-discrete basis of clopen sets , where $\tau$ is the $P$-number of $X$. Then $X$ is a GO-space.
\end{thm}
\begin{proof}
Let $\mathcal B = \bigcup_{\alpha<\tau} \mathcal B_\alpha$  be a basis as in the hypothesis. Since each $\mathcal B_\alpha$ is a discrete family of clopen sets, we may assume that $\mathcal B_\alpha$ is a cover of $X$ for each $\alpha$.  Additionally, we may assume that $\tau$ is infinite and $\mathcal B_0=\{X\}$.

Inductively, for each $\alpha<\tau$, we will define a relation $\mathcal O_\alpha$ on $X$ so that $\bigcup_{\alpha<\tau}\mathcal O_\alpha$ will be an order $\prec$ on $X$ compatible with the topology of $X$. In addition, we will define a collection $\mathcal P_\alpha$ so that  $\bigcup_{\alpha<\tau}\mathcal P_\alpha$ will consist of $\prec$-convex sets and form  a basis for the topology of $X$. 

\par\bigskip\noindent
\underline{\it Step $0$.} Put $\mathcal O_0 = \emptyset$ and $\mathcal P_0=\{X\}$.

\par\bigskip\noindent
\underline{\it Assumption ($\beta<\alpha$).} Assume that for each $\beta <\alpha$, where $\alpha <\tau$, we have defined $\mathcal O_\beta$ and the following hold:
\begin{description}
	\item[{\rm A1}] $\mathcal O_\gamma\subset \mathcal O_\beta$ if $\gamma<\beta$.
	\item[{\rm A2}] If $x,y$ are separated by $\bigcup_{\gamma\leq \beta} \mathcal B_\gamma$, then either $(x,y)$ or $(y,x)$, but not both,  is in $\mathcal O_\beta$.
	\item[{\rm A3}] If $(x,y),(y,z)\in \mathcal O_\beta$, then $(x,z)\in \mathcal O_\beta$.
	\item[{\rm A4}] $(x,x)\not \in \mathcal O_\beta$ for any $x\in X$.
	\item[{\rm A5}] If $x,y$ are not separated by any element of $\bigcup_{\gamma\leq \beta}\mathcal B_\gamma$, then
the following two statements are true:

$$ \forall z [(z,x)\in \mathcal O_\beta\to(z,y)\in \mathcal O_\beta] $$
$$ \forall z [(x,z)\in \mathcal O_\beta\to(y,z)\in \mathcal O_\beta] $$
\end{description}
Note that our assumptions hold for $\beta=0$.

\par\bigskip\noindent
\underline{\it Step $\alpha<\tau$.} Put $\mathcal F_\alpha =\bigcup_{\beta<\alpha}\mathcal B_\beta$.
\par\medskip\noindent
{\it Definition of $\mathcal P_\alpha$:} $P\in \mathcal P_\alpha$ if and only if $P$ is the non-empty intersection of a maximal subfamily of $\mathcal F_\alpha$ with the finite intersection property.

\par\medskip\noindent
The next three claims will be used in our post-construction argument.

\par\bigskip\noindent
{\it Claim 1. Any $I\in  \mathcal P_\alpha$ is open in $X$.}
\par\smallskip\noindent
Let $\mathcal I\subset \mathcal F_\alpha$ be such that $I=\bigcap \mathcal I$. Since each $\mathcal B_\beta$ is discrete, $\mathcal I$ meets each $\mathcal B_\beta$ for $\beta<\alpha$ by exactly one element. Therefore, $|\mathcal I|\leq \alpha<\tau$. Since the $P$-number of $X$ is $\tau$, the set $I=\bigcap \mathcal I$ is open, which proves the claim.

\par\bigskip\noindent
{\it Claim 2. If $x\in B_x\in \mathcal F_\alpha$, then $x\in P\subset B_x$ for some $P\in \mathcal P_\alpha$.}
\par\smallskip\noindent
Since each $\mathcal B_\beta$ is a disjoint cover of $X$, the family $\{B_{x,\beta}:x\in B_{x,\beta}\in \mathcal B_\beta, \beta<\alpha\}$ is a maximal subfamily in $\mathcal F_\alpha$ with the finite intersection property and its intersection $P$ has the desired properties, which completes the claim.

\par\bigskip\noindent
The next claim is obvious and is stated without a proof.
\par\bigskip\noindent
{\it Claim 3. $\mathcal P_\alpha$ is a partition of $X$ inscribed in each $\mathcal B_\beta$, $\beta<\alpha$.}
\par\smallskip\noindent
 
\par\bigskip\noindent
Next order elements of $\mathcal B_\alpha = \{B_{\alpha,\lambda}:\lambda<\tau_\alpha\}$.
\par\medskip\noindent
{\it Definition of $\mathcal O_\alpha$:} 
Put  $\mathcal O_\alpha'=\{(x,y):x,y\in P\in \mathcal P_\alpha, x\in B_{\alpha,\lambda}, y\in B_{\alpha,\gamma},\lambda<\gamma\}$ and
$\mathcal O_\alpha = \mathcal O_\alpha'\cup (\bigcup_{\beta<\alpha}\mathcal O_\beta).$

\par\medskip\noindent
Let us verify A1-A5 for $\alpha$.

\par\bigskip\noindent
\begin{description}
	\item[\underline{\rm A1 check}] This property follows from the fact that each $\mathcal O_\beta$, where $\beta<\alpha$, is represented in the union that defines $\mathcal O_\alpha$.
	\item[\underline{\rm A2 check}] Fix $x,y$ separated by $\bigcup_{\beta\leq \alpha}\mathcal B_\beta$. If $x,y$ are separated by $B\in \mathcal B_\beta$ for some $\beta<\alpha$, then apply A2 for $\beta$ and $A1$ for $\alpha$.

Otherwise, by Claim 3, there exists $P\in \mathcal P_\alpha$ that contains $\{x,y\}$. Since $\mathcal B_\alpha$ is a disjoint cover of $X$ that separates $x$ and $y$, we conclude that $x\in B_{\alpha,\lambda}$ and $y\in \mathcal B_{\alpha,\gamma}$ for some $\lambda\not=\gamma$. By the definition of $\mathcal O_\alpha'$, either $(x,y)$ or $(y,x)$, but not both,  is in $\mathcal O_\alpha'$ and, therefore, in $\mathcal O_\alpha$.
	\item[\underline{\rm A3 check}] Fix $(x,y),(y,z)\in \mathcal O_\alpha$.

Assume that  both $(x,y),(y,z)\in \mathcal O_\beta$ for some $\beta<\alpha$. Then $(x,z)\in \mathcal O_\beta$ by A3 for $\beta$. Then $(x,z)\in \mathcal O_\alpha$ by A1 for $\alpha$.

Assume that  $(x,y)\in \mathcal O_\beta $ and $(y,z)\not\in \mathcal O_\beta$ for some $\beta<\alpha$.  By A2 for $\alpha$, we conclude that $(y,z)\not\in \mathcal O_\beta$. Then $y$ and $z$ are not separated by 
$\bigcup_{\gamma\leq \beta}\mathcal B_\gamma$. Then $(x,z)\in \mathcal O_\beta$ by A5 for $\beta$. Therefore, $(x,z)\in \mathcal O_\alpha$ by A1 for $\alpha$.

Assume that  $(x,y)\not\in \mathcal O_\beta $ and $(y,z)\in \mathcal O_\beta$ for some $\beta<\alpha$. Then the previous argument applies. 

Assume that neither $(x,y)$ nor $(y,z)$ is $\bigcup_{\beta<\alpha}\mathcal O_\beta$. Then there exists $P\in \mathcal P_\alpha$ such that $\{x,y,z\}\subset P$. Also, there exist $\beta<\gamma<\lambda$ such that $x\in B_{\alpha,\beta},y\in B_{\alpha,\gamma},z\in B_{\alpha,\lambda}$. Since $\beta<\lambda$, we conclude that $(x,z)\in \mathcal O_\alpha'\subset \mathcal O_\alpha$.

	\item[\underline{\rm A4 check}] Fix any $x\in X$. By assumption, $(x,x)\not \in \bigcup_{\beta<\alpha}\mathcal O_\beta$. By definition, $(x,x)\not\in \mathcal O_\alpha'$. Therefore, $(x,x)\not\in \mathcal O_\alpha$.
	\item[\underline{\rm A5 check}] Assume that  $x$ and $y$ are not separated by  $\bigcup_{\gamma\leq \alpha}\mathcal B_\gamma$. Then there exist $P\in \mathcal P_\alpha$ such that $\{x,y\}\subset P$ and $\gamma$ such that 
$x,y\in B_{\alpha,\gamma}\in \mathcal B_\alpha$. We will consider the case $(z,x)\in \mathcal O_\alpha$.

Assume $(z,x)\in \mathcal O_\beta$ for some $\beta<\alpha$. Then $(z,y)\in \mathcal O_\beta$  by A5 for $\beta$. By A1 for $\alpha$, we conclude that $(z,y)\in \mathcal O_\alpha$.

Assume  $(z,x)\not\in \mathcal O_\beta$ for any $\beta<\alpha$. Then there exists $P\in \mathcal P_\alpha$ such that $\{z,x,y\}\subset P$. Since $(z,x)\in \mathcal O_\alpha$ and $x\in B_{\alpha,\gamma}$ there exists $\lambda <\gamma$ such that $z\in B_{\alpha,\lambda}$. Since $y\in B_{\alpha,\gamma}$ we conclude that $(z,y)\in \mathcal O_\alpha'\subset \mathcal O_\alpha$.
\end{description}

\par\bigskip\noindent
The inductive construction is complete.

\par\bigskip\noindent
Define $\prec$  by letting $x\prec y$ if and only if $(x,y)\in \bigcup_{\alpha<\tau}\mathcal O_\alpha$.  Let us show that the relation is a linear order. Non-reflexivity follows from A4. Transitivity follows from A3 an A1. To check comparability, fix distinct $x,y\in X$. Since $X$ is Hausdorff and $\mathcal B$ is a basis for the topology of $X$, there exists $\alpha<\tau$ such that
$x$ and $y$ are separated by $\mathcal B_\alpha$. By A2, either $(x,y)$ or $(y,x)$ is in $\mathcal O_\alpha$. 

\par\bigskip\noindent
The conclusion of the theorem follows from the next two claims.

\par\bigskip\noindent
{\it Claim 4. $\bigcup_{\alpha<\tau} \mathcal P_\alpha$ forms a basis for the topology of $X$.}
\par\smallskip\noindent
To prove the claim, fix $x\in X$ and $B\in \mathcal B_\alpha$ containing $x$. By Claim 3, there exists $P\in \mathcal P_{\alpha+1}$ such that $x\in P\in B$. By Claim 1, $P$ is open in $X$. The claim is proved.

\par\bigskip\noindent
{\it Claim 5. Every element of $\bigcup_{\alpha<\tau}\mathcal P_\alpha$ is convex with respect to $\prec$.}
\par\smallskip\noindent
Fix $P\in\mathcal P_\alpha$ and $x,y\in P$ with $x\prec y$.  Fix any $z$ such that $x\prec z\prec y$. 
Since $x,y\in P\in\mathcal P_\alpha$, we conclude that $x,y$ are not separated by $\bigcup_{\beta<\alpha}\mathcal B_\beta$. If $x,z$ were separated by  $\bigcup_{\beta<\alpha}\mathcal B_\beta$ then $(x,z)$ would have been in $\mathcal O_\beta$ for $\beta<\alpha$. By A5, $(y,z)$ would have been in $\mathcal O_\beta$, contradicting $z\prec y$. Therefore, $x,y,z$ are not separated by 
 $\bigcup_{\beta<\alpha}\mathcal B_\beta$. By Claim 3, $x,y,z\in P$, which proves convexity of $P$.
\end{proof}

\par\bigskip
To  state the promised necessary condition  for suborderability of finite powers, we need the following lemma.
\par\bigskip\noindent
\begin{lem}\label{lem:productivity1}
Let $\tau$ be the $P$-number of non-discrete spaces $X$ and $Y$. Let 
$X$ and $Y$  have  $\tau$-discrete bases of clopen sets.  Then the $P$-number of $X\times Y$ is $\tau$, and  $X\times Y$  has a $\tau$-discrete basis of clopen sets.
\end{lem}
\begin{proof}
Let $\{O_\alpha:\alpha<\kappa\}$ be a collection of open sets in $X\times Y$, where $\kappa<\tau$. Assume that 
$S=\bigcap_{\alpha<\kappa}O_\alpha$ is not empty. Fix any $(x,y)\in S$. Then for each $\alpha<\kappa$, there exist $U_\alpha,V_\alpha$ open neighborhoods of $x$ and $y$, respectively, such that $U_\alpha\times V_\alpha\subset O_\alpha$.
Since the $P$-number of $X$ and $Y$ is $\tau$, we conclude that $U=\bigcap_{\alpha<\kappa}U_\alpha$ and $V=\bigcap_{\alpha<\kappa}V_\alpha$ are open. Therefore, $(x,y)\in U\times V\subset S$. Hence $S$ is open. Therefore, the $P$-number of $X\times Y$ is greater than or equal to $\tau$. Since $X\times Y$ contains a copy of $X$, the $P$-number of $X\times Y$ is at most $\tau$. Thus, the $P$-number of $X\times Y$ is $\tau$.

Next, let $\mathcal U_\alpha$ be a discrete collection of clopen subsets of $X$ such that $\bigcup_{\alpha<\tau}\mathcal U_\alpha$ is a basis for the topology of $X$. Similarly, we fix a basis $\bigcup_{\alpha<\tau} \mathcal V_\alpha$ for the topology of $Y$.  Put $\mathcal B_{\alpha\beta} = \{U\times V: U\in \mathcal U_\alpha, V\in \mathcal V_\beta\}$. The family $\mathcal B_{\alpha\beta}$ is a discrete collection of clopen sets, since $\mathcal U_\alpha$ and $\mathcal V_\beta$ are. By the definition of  product topology, $\bigcup_{\alpha,\beta<\tau}\mathcal B_{\alpha\beta}$ is a basis for the topology of $X\times Y$.
\end{proof}
\par\bigskip\noindent
The statement of the next Lemma is a simple corollary to Lemma \ref{lem:productivity1}
\par\bigskip\noindent
\begin{lem}\label{lem:productivity}
Let $\tau$ be the $P$-number of a non-discrete $X$. Let 
$X$ have a $\tau$-discrete basis of clopen sets.  Then for any positive integer $n$, the $P$-number of $X^n$ is $\tau$ and  $X^n$  has a $\tau$-discrete basis of clopen sets.
\end{lem}

\par\bigskip
Theorem \ref{thm:sufficient} and Lemma \ref{lem:productivity} imply our main result of the paper, stated as follows.
\par\bigskip\noindent
\begin{thm}\label{thm:mainpower}
If the  $P$-number of $X$ is $\tau$ and  $X$ has a $\tau$-discrete basis of clopen sets,  then $X^n$ is a generalized ordered space for any natural  number $n$.
\end{thm}

\par\bigskip
Theorem \ref{thm:mainpower} implies, in particular, that if $X$ is zero-dimensional and the weight of $X$ is equal to the $P$-number of $X$, then any finite power of $X$ is suborderable. Since the square of the Sorgenfrey Line is not suborderabe, the weight-P-number equality cannot be replaced by the density-P-number equality. Also, having $P$-number equal to weight locally and uniformly is not sufficient either. The space of countable ordinals serves as a counterexample. 

It is natural to wonder if  our sufficient condition for suborderability of $X^2$ is  a characterization. Our next result shows that it may be, at least for a sufficiently large class of spaces. For the sake of our next result only, we say that $X$ is 
character-homogeneous at points of a set $S\subset X$ if $\chi(x,X)=\chi(y,X)$ for all  $x,y\in S$. We will use the theorem of Engelking and Lutzer that (see \cite{BL} or \cite{Lut}) {\it "A GO-space $X$ is paracompact if and only if no close subspace of $X$ is homeomorphic to a closed subspace of a regular uncountable cardinal"}. 
\par\bigskip\noindent
\begin{thm}\label{thm:ordinal}
Let $X$ be a subspace of an ordinal. Then the following conditions are equivalent:
\begin{enumerate}
	\item $X$ has no stationary subspaces and is character-homogeneous at all non-isolated points.
\item $X$ has a $\tau$-discrete basis of clopen sets, where $\tau$ is the $P$-number of $X$.
	\item $X^n$ is suborderable for any $n$.
\end{enumerate}
\end{thm}
\begin{proof}
{\it Proof of (1)$\Rightarrow$(2)}. We will prove the implication by induction on ordinal $\alpha$ that can host $X$.
\par\bigskip\noindent
\underline {\it Step ($n$ is finite).} Then $X=\{a_1,...,a_m\}$ for some $m\leq n$ and the $P$-number of $X$ is $m$. The family 
$\bigcup_{i\leq m}\mathcal B_i$, where $\mathcal B_i=\{\{x_n\}: n=1,...,m\}$, is an $m$-discrete basis of $X$.

\par\bigskip\noindent
{\it Remark.} Observe that a discrete space of any cardinality has a $\tau$--discrete basis for any $\tau>0$,
\par\bigskip\noindent
\underline {\it Assumption for $\beta<\alpha$.} Assume that for any $X\subset \beta <\alpha$, the implication (1)$\Rightarrow$(2) is true.

\par\bigskip\noindent
\underline{\it Inductive Step $\alpha$.} 

\par\medskip\noindent
\underline {\it Case of limit $\alpha$.} Since $X$ has no subspaces homeomorphic to a stationary subset, by the Engelking-Lutzer Theorem, $X$ is paracompact. Therefore, $X$  can be written as a free sum $\oplus_{\gamma\in \Gamma}X_\gamma$, where $X_\gamma\subset \gamma<\alpha$ for each $\gamma\in \Gamma$. Since $X$ is character homogeneous at all non-isolated points, there exists a cardinal $\tau$ such that
$\chi(x,X)=\tau$ for all non-isolated $x\in X$. Then each $X_\gamma$ has no subspaces homeomorphic to a stationary subset and is character homogeneous at non-isolated points. The $P$-number of any non-discrete $X_\gamma$  is $\tau$. By Inductive Assumption and Remark,  each $X_\gamma$ has a $\tau$-discrete basis $\mathcal B_\gamma$ of clopen sets. Since  $\{X_\gamma:\gamma\in \Gamma\}$ is a discrete cover, the family $\bigcup_{\gamma\in\Gamma}\mathcal B_\gamma$ is a $\tau$-discrete basis of $X$ consisting of clopen sets.  

\par\medskip\noindent
\underline {\it Case of isolated $\alpha$.} Let $\alpha =\beta +1$. We may assume that $\beta\in X$ and $\beta$ is not isolated in $X$. Otherwise, we can reduce our scenario to a smaller ordinal.

Since $X$ has no subspaces homeomorphic to a stationary subset, we conclude that $X\setminus \{\beta\}$ can be written as a free sum $\oplus_{\gamma\in \Gamma}X_\gamma$, where $X_\gamma\subset \gamma<\beta$ for each $\gamma\in \Gamma$
and $|\Gamma |=cf(\beta )$. Since $\beta$ is not isolated in $X$, we conclude that the character of $\beta$ is $cf(\beta)$. Then the character at all non-isolated points of $X$ is $cf(\beta)$. Following the argument of the limit case,  each $X_\gamma$ has a $cf(\beta)$-discrete basis of clopen sets $\mathcal B_\gamma$.  The family 
$$
\{X\setminus (\gamma + 1):\gamma\in \Gamma\}\cup  \{B\in \mathcal B_\gamma:\gamma\in \Gamma\}
$$
is $cf(\beta)$-discrete basis of $X$ consisting of clopen sets.  

\par\bigskip\noindent
{\it Proof of (2)$\Rightarrow$(3).} This implication is Theorem \ref{thm:mainpower}.

\par\bigskip\noindent
{\it Proof of (3)$\Rightarrow$(1).}  If $X$ had a stationary subset or  two limit points of distinct characters then $X\times X$ would not have been hereditarily normal. This statement follows from the argument of Katetov \cite{Kat} that if $X\times Y$ is hereditarily normal then either  every closed subset of $X$ is a $G_\delta$-set or every countable subset of $Y$ is closed. Since every GO space is hereditarily normal, the implication is proved.
\end{proof}
\par\bigskip
We would like to finish the paper with a few questions that naturally arise as a result of our discussion.
\par\bigskip\noindent
\begin{que}
Let $X\times X$ be suborderable. Is $X\times X$ orderable? What if $X$ is orderable?
\end{que}
\begin{que}
Let $X\times X$ be orderable (suborderable). Is $X^n$ orderable (suborderable)?
\end{que}
\begin{que}
Let $X$ have a $\tau$-discrete basis of clopen sets, where $\tau$ is the $P$-number of $X$. Is $X\times X$ orderable?
What if $\tau$ is the weight of $X$?
\end{que}
\par\bigskip\noindent
Finally, the unaccomplished goals of the paper are summarized in the next two questions.
\begin{que}
Assume that $X\times X$ is suborderable (or orderable). Is it true that $X$ has a $\tau$-discrete basis of clopen sets, where $\tau$ is the $P$-number of $X$?
\end{que}
\begin{que}
Assume that $X\times X$ is suborderable (or orderable) space of density $\tau$.  Is it true that the weight of $X$ is equal to the $P$-number of $X$?
\end{que}

\par\bigskip\noindent
{\bf Acknowledgment.} The author would like to thank the referee for valuable remarks and corrections.

\end{document}